\documentclass[12pt,reqno]{amsart}
\usepackage[utf8]{inputenc}
\usepackage[vmargin=2.5cm,hmargin=3cm]{geometry} 
\usepackage[british]{babel}
\usepackage{amsaddr}

\title[The distribution of the multiplicative index]{The distribution of the multiplicative index\\ of algebraic numbers over residue classes}
\date{}

\author[P. Moree, A. Perucca, P. Sgobba]{Pieter Moree}
\address{Max-Planck-Institut f\"ur Mathematik \\ Vivatsgasse 7, D-53111 Bonn, Germany}
\email{moree@mpim-bonn.mpg.de}

\author[]{Antonella Perucca}
\address{Department of Mathematics, University of Luxembourg\\ 6 Avenue de la Fonte, L-4364 Esch-sur-Alzette, Luxembourg}
\email{antonella.perucca@uni.lu}

\author{Pietro Sgobba}
\thanks{\emph{Corresponding author:} Pietro Sgobba}
\address{Department of Pure Mathematics, Xi'an Jiaotong--Liverpool University\\ 111 Ren’ai Road, Suzhou 215123, China}
\email{pietro.sgobba@xjtlu.edu.cn}

\subjclass[2020]{Primary: 11R45; Secondary:  11A07, 11R44} 
\keywords{Reductions of algebraic numbers, multiplicative index and order, primes in arithmetic progression, natural density}

\usepackage{amssymb,latexsym}
\usepackage{comment}
\usepackage{titletoc}
\usepackage{amsfonts, amsmath, amsthm, amssymb, fancybox, graphicx, color, times, babel}
\usepackage[utf8]{inputenc}
\usepackage{multirow}

\usepackage[usenames,dvipsnames]{xcolor}

\usepackage[all,cmtip]{xy}
\newtheorem {thm}{Theorem}
\newtheorem* {thm*}{Theorem}

\newtheorem {cor}[thm]{Corollary}

\newtheorem* {cor*}{Corollary}
\newtheorem {lem}[thm]{Lemma}

\newtheorem {prop}[thm]{Proposition}

\newtheorem {rem}[thm]{Remark}

\theoremstyle{definition}
\newtheorem {exa}[thm]{Example}
\newtheorem* {conj*}{Conjecture}

\newtheorem* {quest*}{Question}

\newcommand{\mc}{\mathcal}
\newcommand{\p}{\mathfrak{p}}
\newcommand{\Q}{\mathbb{Q}}
\newcommand{\Z}{\mathbb{Z}}
\DeclareMathOperator{\dens}{dens}
\DeclareMathOperator{\ind}{ind}
\DeclareMathOperator{\ord}{ord}
\DeclareMathOperator{\Li}{Li}
\DeclareMathOperator{\Gal}{Gal}
\DeclareMathOperator{\Frob}{Frob}
\DeclareMathOperator{\id}{id}
\DeclareMathOperator{\N}{N}
\DeclareMathOperator{\lcm}{lcm}
\newcommand{\bo}[1]{O\left( #1 \right)}
\newcommand{\av}[1]{\left\lvert#1\right\rvert}

\begin{document}

\maketitle
\begin{abstract}
Let $K$ be a number field and $G$ a finitely generated torsion-free subgroup of $K^\times$. Given a prime $\p$ of $K$ 
we denote by $\ind_\p(G)$ the index of the subgroup $(G\bmod\p)$ of the multiplicative group of the residue field at $\p$. Under the Generalized Riemann Hypothesis we determine the natural density of primes of $K$ for which this index is in a prescribed set $S$ and has prescribed Frobenius in a finite Galois extension $F$ of $K$. We study in detail the natural density in 
case  $S$ is an arithmetic progression, in particular its positivity.
\end{abstract}

\section{Introduction}
The distribution of the multiplicative index of an integer seems to have
been first studied by Pappalardi \cite{pappa} in 1995. Under the Generalized Riemann Hypothesis (GRH) he provided
asymptotic formulae for $\sum_{p\le x}f(\ind_p(g))$, for $f$ satisfying  fairly mild restrictions (here and in the sequel 
we denote the rational primes by $p$). This line of 
investigation was continued in 2012 by Felix and Murty 
\cite{FM} and later by Felix for higher rank in  \cite{Felixhigher}.
Given a set of integers $S$ and a natural number $g$, in \cite{FM} it was proven that
\begin{equation}\label{FelixMurty}
\pi_{g,S}(x)
:=\av{\{p\le x:\ind_g(p)\in S\}}=c_{g,S}\Li(x)+O\Big(\frac{x}{(\log x)^{2-\epsilon}}\Big),
\end{equation}
where $c_{g,S}$ is a constant defined by a series whose terms depend on the set $S$,
$\Li(x):=\int_2^xdt/\log t$ denotes the logarithmic integral and $\epsilon>0$ is arbitrary. It is a difficult problem to determine whether
$c_{g,S}$ is positive or not, cf.\,Felix \cite{Felix}. The special case
where $S$ is an arithmetic progression was already considered by Moree \cite[Thm.\,5]{moree} in 2005. For example, he proved 
Theorem 
\ref{densityaslinearcombination} below in case $G=\langle g\rangle$, $F=K=\mathbb Q$.
\par In this paper we consider the behavior of 
$\pi_{g,S}(x)$, with $\mathbb Q$ replaced by a number field $K$ and $g$ by a finitely generated torsion-free 
subgroup $G$ of $K^\times$. 
Instead of over rational primes, we sum now over primes $\p$ of norm $\le x$. Under GRH we establish 
in Theorem \ref{thm-index}, see Section \ref{sec:two},
an asymptotics similar to \eqref{FelixMurty}, but with a weaker error term depending on the rank of $G$. 
Notice that our result relies on variations for number fields of Hooley's proof of Artin's primitive root conjecture under the assumption of GRH \cite{hooley}.
In Section \ref{sec:APdensity} we then restrict to the case where $S$ consists of integers in an arithmetic progression $a\bmod d$. In 
Theorem \ref{densityaslinearcombination} we show that in this case the natural density can be expressed as a linear combination of
at most $\varphi(d')-1$ Artin-type constants, with $d'=d/(a,d)$. 
The positivity of the density is studied in Section \ref{sec:vanishing}, the numerical evaluation of the Artin-type
constants in 
Section \ref{EPconstants}. In the final section we 
demonstrate our results by determining the density for two examples and compare the outcome with
an experimental approximation.
\par We take $G$ to be fixed, but one can also ask what happens for a ``typical" $G$.
Ambrose \cite{ambrose} considered the average index of the group generated by a finite number of elements in the residue field at a prime of a number field and provided asymptotic formulae for the average order of this quantity.
\par Likewise we can wonder about the above questions, but for the multiplicative order, rather than the 
index. As far as the authors know, these were first studied by 
Chinen and Murata \cite{CM} for $d = 4$, and a little later
by Moree by a simpler method.
Both Chinen and Murata, and independently Moree, went on to
write various further papers (he surveyed his results in \cite{Moreeannounce}). Under an appropriate generalization of the Riemann Hypothesis it turns out that the natural density of primes $p\le x$ such
that the multiplicative order of $g$ modulo $p$ is 
congruent to $a\bmod d$ exists.
Denote it by $\delta_g(a,d)$ and the associated counting function
by $N_g(a,d)(x)$.
The proof of the existence of $\delta_g(a,d)$  by Moree is based on the identity
$$N_g(a,d)(x)=\sum_{t=1}^{\infty}\av{\{p\le x\,:\,\ind_p(g)=t,~p\equiv 1+ta \bmod dt\}}.$$
The average density of elements of order congruent to 
$a\bmod d$ in a field of prime
characteristic also 
exists, but is a much simpler quantity, see Moree \cite{moreeFF}. It has very similar features to
$\delta_g(a,d)$.
\par In the special case where $d$ divides $a$, 
we are just asking for the density of primes $p$ such
that $d$ divides the multiplicative order of $g$ modulo $p$. This density is much easier to
deal with and turns out to be a rational number. 
This can be proven unconditionally, 
see for example \cite{MoreeFACM,Wie1,Wie2}.  
\par Ziegler \cite{zieg}, using the approach of
Moree, was the first to study the order in
arithmetic progression problem in the setting of 
number fields. His work was generalized by Perucca and Sgobba in \cite{PeruccaSgobba,PeruccaSgobba2}, who obtained in particular uniformity results for the distribution of the order. It is expected that, likewise, some uniformity also holds for the distribution of the index into suitably related congruence classes, however at the moment it is not clear how to obtain such a result. For example, it does not follow from \cite[Cor.~5.2]{PeruccaSgobba}, in spite of the fact that congruence conditions on both the order and the size of the multiplicative group lead to congruence conditions on the index. We leave this as a research direction and as an open problem to the reader.

It is also still unknown whether this kind of results can be proven unconditionally. Although the results of this paper mostly rely on GRH,  there are fundamental papers on Artin's conjecture for primitive roots providing unconditional results, see for example \cite{guptamurty} by Gupta and Murty, and  \cite{HB} by Heath-Brown. 
However, the contrast between what can be proven conditionally versus unconditionally in this area is quite dramatic. We note though that the infinitude of primes $p$ in a prescribed arithmetic progression with $\ind_p(g)\ne t$, with $t$ prescribed, \emph{can} be unconditionally determined (however, not its density) \cite{MoreeSha}.
Last but not least, Pappalardi obtained  quantitative results without relying on GRH under certain convergence conditions, see for example \cite[Thm.~1]{pappa}, from which one can determine the density for $\ind_p(g)$ being squarefree (and, more generally, $k$-free with $k\geq2$).

\section{The existence of the density of primes with prescribed index and Frobenius}
\label{sec:two}
Let $K$ be a number field, and $F/K$ a finite Galois extension. Let $G$ be a finitely generated and torsion-free subgroup of $K^\times$ having positive rank $r$.
Our goal is to determine the density of the set $P$ of primes $\p$ of $K$ (defined in the next theorem) with prescribed
index and Frobenius.
The notation $F,K,G$ and $r$ will be maintained throughout. We also set $K_{m,n}:=K(\zeta_m,G^{1/n})$ for $m\mid n$, and similarly for $F_{m,n}$. 
Further we make use of the following usual notation: $\zeta_n$ denotes an $n$-th primitive root of unity, $\mu$ the Möbius function, 
and $\varphi$ Euler's totient function. We write $\log^a x$ as shorthand for $(\log x)^a$, and
$(a,b)$ for $\gcd(a,b)$.
\par We recall that Landau's prime ideal theorem states that
\begin{equation}
\label{PIT}    
\av{\left\{\p:\N\p\leq x\right\}}=\Li(x)+O_{K}(xe^{-c_{K}{\sqrt {\log x }}}),
\end{equation}
where $c_K>0$ is a constant depending on $K$. 
\begin{thm}\label{thm-index} {\rm (under GRH)}.
Let $K$ be a number field, and let $G$ be a finitely generated and torsion-free subgroup of $K^\times$ of positive rank $r$. 
Let $F/K$ be a finite Galois extension, and let $C$ be a union of conjugacy classes in $\Gal(F/K)$. Let $S$ be a non-empty set of positive integers.
Define
\[ P:=\{\p : \ind_\p(G)\in S,\, \Frob_{F/K}(\p)\in C \}, \]
where $\p$ ranges over the primes of $K$ unramified in $F$ and for which $\ind_\p(G)$ is well-defined.
We let $P(x)$ be the number of prime ideals in $P$ of norm $\le x$.
We have the asymptotic estimate
\[  P(x)  
 = \frac{x}{\log x}\sum_{t\in S}\,\sum_{v=1 }^\infty\frac{\mu(v)c(vt)}{[F_{vt,vt}:K]}+
 {\bo{\frac{x}{\log^{2-\frac{1}{r+1}}x}}}, \\
 \]
 where 
 \[ c(n)=\av{\{\sigma\in\Gal(F_{n,n}/K) : \sigma|_{K_{n,n}}=\id,\, \sigma|_F\in C  \}}\,. \]
 The implicit constant in the $O$-term depends only on $K$, $F$ and $G$.
 \end{thm}
\par This result in combination with the prime ideal theorem leads to the following corollary.
\begin{cor}
{\rm(under GRH)}. Let $S$ be a non-empty set of positive integers. The natural density of the primes $\p$ of $K$ such that $\ind_\p(G)\in S$ and $\Frob_{F/K}(\p)\in C$ exists and  is given by
\[ \sum_{t\in S} \ \sum_{v=1}^\infty\frac{\mu(v)c(vt)}{[F_{vt,vt}:K]}\,. \]
\end{cor}
\par We will now formulate some preliminaries required
for the proof of Theorem \ref{thm-index}. Our starting point is \cite[Prop.~5.1]{PeruccaSgobba},
which was established  for rank 1 in \cite[Prop.~1]{zieg}.
\begin{thm}
\label{thm-Rt}
{\rm(under GRH)}.
For $x\geq t^{3}$, the number $R_t(x)$ of primes $\p$ with norm up to $x$, unramified in $F$, and such that $\ind_\p(G)=t$ and $\Frob_{F/K}(\p)\in C$  satisfies
\[ R_t(x)=\Li(x)\sum_{v=1}^\infty\frac{\mu(v)c(vt)}{[F_{vt,vt}:K]}+O\left(\frac{x}{\log^2 x}\right)
+O\left(\frac{x\log \log x}{\varphi(t)\log^2 x}\right). \]
The implicit constant in the $O$-term depends only on $K$, $F$ and $G$.
\end{thm}
\par The following lemma is a straightforward generalization of \cite[Lem.~6]{moree}, taking into account that 
for every natural number $n$, the ratio
\begin{equation}
\label{Cndef}    
C(n):=\frac{\varphi(n)n^r}{[K_{n,n}:K]}
\end{equation}
is bounded above by some constant $D$, depending only on $K$ and $G$ (see  \cite[Thm.~1.1]{PeruccaSgobba}).
\begin{lem}\label{lem-zieg}
For every real number $y\geq1$ we have
\[  \sum_{t\le y}\,\sum_{n=1}^\infty\frac{\mu(n)}{[K_{nt,nt}:K]}=1+\bo{\frac{D}{y^r}},   \]
where the  implicit constant is absolute.
\end{lem}
\begin{proof}
We claim that 
$$\sum_{n>y}\frac{1}{n^r\varphi(n)}=O\Big(\frac{1}{y^r}\Big).$$
For $r=1$ this is due to Landau \cite{Landau}, who
first proved that 
\begin{equation}
\label{landausum}    
\sum_{n\le x}\frac{1}{\varphi(n)}=A\log x+B+O\Big(\frac{\log x}x\Big),
\end{equation}
with
$A$ and $B$ explicit constants, and then applied partial integration. The
proof for arbitrary $r$ is completely analogous. Since  
$\varphi(nt)\ge \varphi(n)\varphi(t)$, we obtain
\begin{equation}
\label{sharperestimate}
\sum_{t>y}\sum_{n=1}^\infty\frac{1}{[K_{nt,nt}:K]}
\le D \sum_{t>y}\frac{1}{t^r\varphi(t)}\sum_{n=1}^{\infty}\frac{1}{n^r\varphi(n)}\ll \sum_{t>y}\frac{D}{t^r\varphi(t)}\ll \frac{D}{y^r},
\end{equation}
where we used that the fourth sum is bounded above by a constant not depending on $r$.
The estimate \eqref{sharperestimate} shows that the double sum in the
statement of Lemma 
\ref{lem-zieg} is absolutely convergent for all $y$.
Thus, we may rearrange the double sum as follows:
\begin{align*}
 \sum_{t=1}^\infty\sum_{n=1}^\infty\frac{\mu(n)}{[K_{nt,nt}:K]} & =  \sum_{m=1}^\infty\sum_{s|m}\frac{\mu(m/s)}{[K_{m,m}:K]}
  =  \sum_{m=1}^\infty\sum_{d|m}\frac{\mu(d)}{[K_{m,m}:K]} \\
  & = \sum_{m=1}^\infty\frac{1}{[K_{m,m}:K]}\sum_{d|m}\mu(d)  =  \frac{1}{[K_{1,1}:K]}=1,
\end{align*}
completing the proof.
\end{proof}

The following is a generalization of Ziegler \cite[Lem.~13]{zieg}.
\begin{lem}\label{lem13} 
{\rm(under GRH)}. We have
\[
\av{\left\{\p:\N\p\leq x,\,\ind_\p(G)>(\log x)^{\frac{1}{r+1}}\right\}}=\bo{\frac{x}{\log^{2-\frac{1}{r+1}}x}},
\]
where the primes $\p$ of $K$ are restricted to
those for which $\ind_\p(G)$ is well-defined. The 
 implicit constant in the $O$-term depends only on $K$ and $G$.
\end{lem}

\begin{proof} The number of primes with ramification index or residue class degree at least $2$ 
is of order $O(\sum_{p\le \sqrt{x}}1)=\bo{\sqrt{x}/\log x}$. We make use of the functions $R_t(x)$ from Theorem \ref{thm-Rt} with $F=K$. 
For any real number $y\geq1$, let $\mc E_y(x)$ be the number of primes $\p$ with $\N\p\leq x$ and such that $\ind_\p(G)>y$.
Notice that
\[ \mc E_y(x)=\av{\left\{\p:\N\p\leq x\right\}}-\sum_{t\le y}  R_t(x)+\bo{\frac{\sqrt{x}}{\log x}}\,. \]
Landau's prime ideal theorem \eqref{PIT} implies the (much) weaker estimate
\begin{equation}
\label{pit}    
\av{\left\{\p:\N\p\leq x\right\}}=\Li(x)+\bo{\frac{x}{\log^2x}}=\frac{x}{\log x}+\bo{\frac{x}{\log^2x}},
\end{equation}
which is all we need.
By Theorem \ref{thm-Rt} and Lemma \ref{lem-zieg} we obtain
\begin{align*}
\sum_{t\le y} R_t(x) &= \Li(x)\sum_{t\le y}\sum_{n=1}^\infty\frac{\mu(n)}{[K_{nt,nt}:K]} + \bo{\frac{xy}{\log^2x}}\\
& + \bo{\frac{x\log\log x}{\log^2x}\sum_{t\le y}\frac{1}{\varphi(t)}}\\
 & =   \Li(x) +\bo{\frac{x}{y^r\log x}}  + \bo{\frac{xy}{\log^2x}}\\ 
 & + \bo{\frac{x\log\log x}{\log^2x}\sum_{t\le y}\frac{1}{\varphi(t)}}\,.
\end{align*}
On taking $y=(\log x)^{1/(r+1)}$, we now obtain
on invoking \eqref{pit} and \eqref{landausum}, the estimate
$$\mc E_y(x)=\bo{\frac{x}{y^r\log x}}+\bo{\frac{xy}{\log^2x}} + \bo{\frac{x(\log\log x)^2}{\log^2x}}=\bo{\frac{x}{\log^{2-\frac{1}{r+1}}x}},$$
completing the proof.
\end{proof}

\begin{proof}[Proof of Theorem \ref{thm-index}]
Set $\rho=2-\frac{1}{r+1}$.
Lemma \ref{lem13} with $y=(\log x)^{\frac{1}{r+1}}$ yields
$$
 P(x) = \sum_{\substack{t\leq y \\ t\in S}}  R_t(x)+\bo{\frac{x}{\log^{\rho}x}}.
$$ 
Estimating the sum as in the proof of Lemma \ref{lem13}, we obtain
\[
P(x)=\Li(x)\sum_{\substack{t\leq y \\ t\in S}}\sum_{v=1 }^\infty\frac{\mu(v)c(vt)}{[F_{vt,vt}:K]}+\bo{\frac{x}
{\log^{\rho}x}}\,.
\]
Now we focus on the main term. We have 
\[ \bigg\rvert\sum_{t\in S}\sum_{v=1 }^\infty\frac{\mu(v)c(vt)}{[F_{vt,vt}:K]}-
	\sum_{\substack{t\leq y \\ t\in S}}\sum_{v=1 }^\infty\frac{\mu(v)c(vt)}{[F_{vt,vt}:K]} \bigg\rvert \leq \sum_{t> y }\sum_{v=1 }^\infty\frac{1}{[F_{vt,vt}:F]}  \]
By \eqref{sharperestimate} the right-hand side is bounded by $\ll y^{-r}$. 
Using this estimate the proof is easily completed.
\end{proof}

\section{The distribution of the index over residue classes}
\label{sec:APdensity}
Let $a,d$ be integers with $d\geq2$.
 We study the density $\dens_G(a,d)$ of primes $\p$ of $K$ such that $\ind_\p(G)\equiv a\bmod d$. 
Under GRH, by Theorem \ref{thm-index} this density exists and we have 
\begin{equation}
\label{densdoublesum}
 \dens_G(a,d)=\sum_{t\equiv a\bmod d}\ \sum_{v\geq1}\frac{\mu(v)}{[K_{vt,vt}:K]}\,. 
\end{equation}
The goal of this section is 
to prove Theorem \ref{densityaslinearcombination}, 
which expresses $\dens_G(a,d)$ as a finite sum of terms depending on Dirichlet characters $\chi$ 
of modulus $d$. These terms involve Artin-type constants $B_{\chi}(r)$ that can be evaluated with multi-precision using
Theorem \ref{numeriek}, thus allowing
one to evaluate $\dens_G(a,d)$ with multi-precision.

We start by explaining our notation.
Given an integer $n\geq1$ we let $G_n$ be the group of characters defined on $(\Z/n\Z)^\times$, so that $G_n\cong(\Z/n\Z)^\times$. For a Dirichlet character $\chi$ we denote by $h_\chi$ the (Dirichlet) convolution $\mu*\chi$ of the M\"obius function $\mu$ with $\chi$, that is $\mu*\chi(n)=\sum_{d\mid n}\mu(d)\chi(n/d)$. Recall that the Dirichlet convolution of two multiplicative functions is again a multiplicative function.

Put $w=\gcd(a,d), a'=a/w$ and $d'=d/w$. 
The integers $t\equiv a\bmod d$ are of the form
$wt'$ with  $t'\equiv a'\bmod d'$. Thus we can rewrite 
$\dens_G(a,d)$ as 
$$\dens_G(a,d)=\sum_{t\equiv a'\bmod d'}\ \sum_{v\geq1}\frac{\mu(v)}{[K_{vwt,vwt}:K]}\,.$$
This expression on its turn can be rewritten as
\begin{align}
\label{hchi}
\dens_G(a,d) & =\sum_{t\equiv a'\bmod d'}\ \sum_{\substack{v_1\geq1\\ t\mid v_1}} \frac{\mu(v_1/t)}{[K_{v_1w,v_1w}:K]}\nonumber	\\
&=\sum_{v_1\geq1}\ \sum_{\substack{t\equiv a'\bmod d'\\ t\mid v_1}} \frac{\mu(v_1/t)}{[K_{v_1w,v_1w}:K]}\nonumber	\\
&=\sum_{v_1\geq1}\left(\frac{1}{\varphi(d')}\sum_{\chi\in G_{d'}} \overline{\chi(a')}h_\chi(v_1) \right)\frac{1}{[K_{v_1w,v_1w}:K]}\nonumber \\
&=\frac{1}{\varphi(d')}\sum_{\chi\in G_{d'}} \overline{\chi(a')}\sum_{v_1\geq1}\frac{h_\chi(v_1)}{[K_{v_1w,v_1w}:K]}\,.
\end{align}
In the second step we used that the double series is absolutely convergent (see the proof of Lemma 
\ref{lem-zieg}). 
In the third step we used \cite[Lem.~9]{moree}, where $\chi$ runs over the Dirichlet characters modulo $d'$.

We now focus on the final sum in \eqref{hchi}. Recall the 
definition \eqref{Cndef} of $C(n)$. By Perucca et al.\,\cite[Thm.~1.1]{PST4}
there exists an integer $n_0$ (depending only on $G$ and $K$) such that
\begin{equation}
\label{n0} 
C(n)=C(\gcd(n,n_0))\,.
\end{equation}
One can easily show that for $m\mid n$, one has $C(m)\mid C(n)$, and hence $n_0$ can be taken to be the minimal integer satisfying
$$C(n_0)=\max_{n\geq 1} \frac{\varphi(n)n^r}{[K_{n,n}:K]}\,.$$
By \eqref{n0} we have
\[ \frac{1}{[K_{n,n}:K]}=\frac{C(\gcd(n,n_0))}{\varphi(n)n^r}\,, \]
and therefore
\[ \sum_{n\geq1}\frac{1}{[K_{n,n}:K]}=\sum_{g\mid n_0}\sum_{\substack{n\geq1\\ (n,n_0)=g}}\frac{C(g)}{\varphi(n)n^r}\,. \]
In our case,
\begin{equation}
\label{vw}    
\sum_{v\geq1}\frac{h_\chi(v)}{[K_{vw,vw}:K]}
	=\sum_{g\mid n_0}\sum_{\substack{v\geq1\\ (vw,n_0)=g}}\frac{C(g)h_\chi(v)}{\varphi(vw)v^rw^r}\,. 
\end{equation}	
If $\sum_{v\ge 1}f(v)$ is some absolute convergent series, we have
\begin{align*}
\sum_{\substack{v\geq1\\ (vw,n_0)=g}}f(v)&=\sum_{\substack{v\geq1\\ (\frac{vw}g,\frac{n_0}g)=1}}f(v)=\sum_{v\ge 1}f(v)\sum_{n\mid \frac{n_0}g,\,n\mid \frac{vw}g}\mu(n)\\
&=\sum_{n\mid \frac{n_0}{g}}\mu(n)\sum_{\substack{v\ge 1\\ n\mid \frac{vw}g}}f(v)=\sum_{n\mid \frac{n_0}{g}}\mu(n)\sum_{\substack{v\geq1\\ \frac{gn}{(gn,w)}\mid v}}f(v),
\end{align*}
where we used that $n$ divides the integer $vw/g$ if and only if 
$gn/(gn,w)$ divides $v$. Thus, in particular,
$$
\sum_{\substack{v\geq1\\ (vw,n_0)=g}}\frac{C(g)h_\chi(v)}{\varphi(vw)v^rw^r}=\frac{C(g)}{w^r}\sum_{n\mid \frac{n_0}{g}}\mu(n)\sum_{\substack{v\geq1\\ \frac{gn}{(gn,w)}\mid v}}\frac{h_\chi(v)}{\varphi(vw)v^r}\,.$$
Inserting the right-hand side into \eqref{vw} and inserting the resulting expression into \eqref{hchi} yields
\[  \dens_G(a,d)=\frac{1}{\varphi(d')}\sum_{\chi\in G_{d'}} \overline{\chi(a')}\sum_{g\mid n_0}\frac{C(g)}{w^r}\sum_{n\mid \frac{n_0}{g}}\mu(n)\sum_{\substack{v\geq1\\ \frac{gn}{(gn,w)}\mid v}}\frac{h_\chi(v)}{\varphi(vw)v^r} \,.   \]
Denoting
$$C_{\chi}(N,w,r)=\sum_{v\ge 1\atop N\mid v}
\frac{h_{\chi}(v)}{\varphi(vw)v^r},$$
we can write this as 
\begin{equation}
    \label{wanteddensity}
\dens_G(a,d)=\frac{1}{\varphi(d')}\sum_{\chi\in G_{d'}} \overline{\chi(a')}\sum_{g\mid n_0}\frac{C(g)}{w^r}\sum_{n\mid \frac{n_0}{g}}\mu(n)\,C_{\chi}\Big(\frac{gn}{(gn,w)},w,r\Big) \,.   
\end{equation}

Let $\kappa(n)=\prod_{p\mid n}p$ denote the squarefree kernel of $n.$
Recall that $h_{\chi}=\mu* \chi$. The following result is a special case
of \cite[Lem.~10]{moree} and expresses $C_{\chi}(N,w,r)$ as an Euler product.
\begin{lem}
\label{productevaluation}
We have
$$
C_{\chi}(N,w,r)=c_{\chi}(N,w,r)B_{\chi}(r),
$$
where  
$$c_{\chi}(N,w,r)=\frac{h_{\chi}(N)\kappa(Nw)}{N^{r+1}w}\prod_{p\mid N}\frac{p^{r+1}}{p^{r+2}-p^{r+1}-p+\chi(p)}
\prod_{p\nmid N\atop p\mid w}\frac{p^{r+1}-1}{p^{r+2}-p^{r+1}-p+\chi(p)},$$
and
\begin{equation}
    \label{Bchidef}
  B_{\chi}(r)=\prod_p\Big(1+\frac{p(\chi(p)-1)}{(p-1)(p^{r+1}-\chi(p))}\Big)\,,
\end{equation}
where $p$ runs over all rational prime numbers.
\end{lem}
\begin{cor}
We have $C_{\chi}(1,1,r)=\sum_{v\geq1}\frac{h_\chi(v)}{v\varphi(v)}=B_{\chi}(r)$.
\end{cor}

\begin{proof}[Proof of Lemma \ref{productevaluation}]We distinguish two cases:\\
\noindent a) The case where $h_{\chi}(N)=0.$\\
\indent We have to verify that $C_{\chi}(N,w,r)=0.$ Since $h_\chi$ is multiplicative and we have $h_{\chi}(p^k)=\chi(p)^{k-1}(\chi(p)-1),$ it follows that if $h_{\chi}(N)=0,$ then
there is a prime divisor $p$ of $N$ with $\chi(p)=1.$ Hence, $h_{\chi}(v)=0$ for
all $v$ that are divisible by $N$ and so $C_{\chi}(N,w,r)=0.$\\
b) The case where $h_{\chi}(N)\ne 0.$\\
\indent We rewrite $C_{\chi}(N,w,r)$ as 
\begin{equation}
    \label{mademultiplicative}
C_{\chi}(N,w,r)=\frac{h_{\chi}(N)}{\varphi(Nw)N^r}
\sum_{v\ge 1}\frac{h_{\chi}(Nv)\varphi(Nw)}{h_{\chi}(N)\varphi(Nvw)v^r},
\end{equation}
and note that the argument is a multiplicative function in $v.$
We apply the Euler product identity to evaluate the sum and obtain
$$\prod_{p\mid N}\frac{p^{r+1}}{p^{r+1}-\chi(p)}
\prod_{p\nmid N\atop p\mid w}\Big(1+\frac{(\chi(p)-1)}{p^{r+1}-\chi(p)}\Big)
\prod_{p\nmid Nw}\Big(1+\frac{p(\chi(p)-1)}{(p-1)(p^{r+1}-\chi(p))}\Big),$$
which can be rewritten as
$$B_{\chi}(r)\prod_{p\mid N}\frac{p^{r+1}(p-1)}{p^{r+2}-p^{r+1}-p+\chi(p)}
\prod_{p\nmid N\atop p\mid w}\frac{(p-1)(p^{r+1}-1)}{p^{r+2}-p^{r+1}-p+\chi(p)}\,.$$
On inserting this in \eqref{mademultiplicative} and noting that
$$\varphi(\kappa(Nw))=\prod_{p\mid N}(p-1)\prod_{p\nmid N\atop p\mid w}(p-1),\quad\frac{\varphi(\kappa(Nw))}{\varphi(Nw)}=\frac{\kappa(Nw)}{Nw},$$ the proof 
is completed.
\end{proof}

\par The density $\dens_G(a,d)$ can be expressed as a finite linear combination
involving the constants $B_{\chi}(r)$. 
Our result generalizes \cite[Thm. 5]{moree} by Moree, who
dealt with the case $F=K=\mathbb Q$ and $G$ of rank $1$.
\begin{thm} {\rm(under GRH)}.
\label{densityaslinearcombination}
Let $a$ and $d$ be two natural numbers. Put $d'=d/(a,d)$.
Assuming that the function $C(n)$, defined in \eqref{Cndef}, is explicitly given, 
we can write
$$\dens_G(a,d)=\sum_{\chi\in G_{d'}}d_{\chi}B_{\chi}(r),$$
with the $d_{\chi}$'s explicit complex numbers (they can be determined using \eqref{wanteddensity} and  Lemma  \ref{productevaluation}).
\end{thm}
The equality of the series given in \eqref{densdoublesum} and the linear combination of Theorem \ref{densityaslinearcombination} is unconditional: it is establishing that these two quantities are densities that requires assuming GRH.

\begin{proof}
In the identity \eqref{wanteddensity} for $\dens_G(a,d)$ we make the substitution
\[ C_\chi\bigg( \frac{gn}{(gn,w)},w,r \bigg)=c_\chi\bigg( \frac{gn}{(gn,w)},w,r \bigg)B_\chi(r) \, \]
(which is allowed by  Lemma  \ref{productevaluation}). The constants $d_\chi$ are obtained by factoring out the terms $B_\chi(r)$, so that for each $\chi\in G_{d'}$ we have 
\[ d_\chi= \frac{\overline{\chi(a')}}{\varphi(d')}\sum_{g\mid n_0}\frac{C(g)}{w^r}\sum_{n\mid \frac{n_0}{g}}\mu(n)c_{\chi}\Big(\frac{gn}{(gn,w)},w,r\Big) \,. \qedhere  
\]
\end{proof}
\subsection{Generic aspects of the behaviour of $\dens_G(a,d)$}
Generically the degree $[K_{vt,vt}:K]$ equals $vt\varphi(vt)$ if $G$ has rank $1$. If every degree occurring
in \eqref{densdoublesum}  satisfied this, then we would obtain
\[  
 \rho(a,d):=\sum_{t\equiv a\bmod d}\ \sum_{v\geq1}\frac{\mu(v)}{vt\varphi(vt)}\,.
\]
The inner sum is easily seen to equal $A\cdot r(t)$, with
$$r(t)=\frac{1}{t^2}\prod_{p\mid t}\frac{p^2-1}{p^2-p-1}\,.$$
Thus we can alternatively write
$$\rho(a,d)=A_1\sum_{t\equiv a\bmod d}r(t)\,,$$ 
with
\begin{equation}
\label{constantAr}    
A_r:=\prod_p\Big(1-\frac{1}{p^r(p-1)}\Big)
\end{equation}
the \emph{rank $r$ Artin constant}. The ``incomplete" rank $r$ Artin constant, defined by restricting to $p$ odd, appears also in other works, such as in Pappalardi \cite{pappa-rank}.
For every $B>0$ we have, see \cite[Thm.~4]{moreeFF},
$$\sum_{p\le x}\rho(p;a,d)=\rho(a,d)\Li(x)+O\Big(\frac{x}{\log^B x}\Big),$$
with $\rho(p;a,d)$ the density of elements of $\mathbb F_p^\times$ having
index congruent to $ a\bmod d$. Thus on average a finite field of prime order has
$\rho(a,d)$ elements having index congruent to $ a\bmod d$. 
Two cases are particularly easy.
\begin{prop}[{\cite[Prop.~4]{moreeFF}}]
\label{null}
One has 
\[ \rho(0,d)=\frac{1}{d\varphi(d)} \quad \text{ and } \quad
\rho(d,2d)=
\begin{cases}
\rho(0,2d) & \text{ if } $d$ \text{ is odd}; \\
3\rho(0,2d) & \text{ if } $d$ \text{ is even}.
\end{cases} \]
\end{prop}
In the remaining cases it is not difficult to express
$\rho(a,d)$ in terms of the $B_{\chi}(1)$'s, see \cite[Prop.\,6]{moreeFF}. When $(a,d)=1$, this
expression takes a particularly simple form, namely
\begin{equation}
    \label{rhoB}
    \rho(a,d)=\frac{1}{\varphi(d)}\sum_{\chi \bmod d}\overline{\chi(a)}B_{\chi}(1)\,.
\end{equation}
In the examples in Section \ref{sec:examples} we will meet $\rho(a,d)$ again.

\section{The positivity of $\dens_G(a,d)$}
\label{sec:vanishing}
 As in the previous section we consider a number field $K$, a finitely generated and torsion-free subgroup $G$ of $K^\times$, and the natural density $\dens_G(a,d)$ of the primes $\p$ of $K$ such that $\ind_\p(G)\equiv a \bmod d$. We are interested in characterizing when this density is positive. 
 
\begin{exa}
Recall that for a prime $\p$ of $K$ of degree $1$ such that $\ind_\p(G)$ is well-defined, we have $d\mid \ind_\p(G)$ if and only if $\p$ splits completely in $K_{d,d}$ (cf.\ \cite[Lem.~2]{zieg}). So by Chebotarev's density theorem we have (without relying on GRH)
\begin{equation}
\label{simpledegree}    
\dens_G(0,d)=\frac{1}{[K_{d,d}:K]}>0. 
\end{equation}
\end{exa}
In view of the above example, we may suppose in the following that $0<a<d$.

We denote by 
$\dens_G(h)$ the density of primes $\p$ such that $\ind_\p(G)=h$, with $h$ a prescribed 
integer, and by $n_0$ an integer satisfying $C(n)=C(\gcd(n,n_0))$ for all $n\geq1$, where $C(n)$ was defined in \eqref{Cndef}. With this notation we are ready to recall the following result by Järviniemi and Perucca:
\begin{thm}[{\cite[Main Thm.\ and Rem.\,4.2]{JP}}, {\rm under GRH}]\label{thm-positive}
The density $\dens_G(h)$ is well-defined for all $h\geq 1$, and we have $\dens_G(h)>0$ if and only if $\dens_G(\gcd(h,n_0))>0$. For any set $S$ of positive integers the following holds: if the density of primes $\p$ of $K$ such that $\ind_\p(G)\in S$ is positive, then there is some $h\in S$ such that $\dens_G(h)>0$.
\end{thm}

\begin{prop}\label{prop_coprime}
{\rm (under GRH)}.
If $d\geq2$ is coprime to $n_0$, then $\dens_G(a,d)>0$.
\end{prop}
\begin{proof}
By Theorem \ref{thm-positive} (taking $S$ to be the set of all positive integers) we know that there is some $h\geq 1$ such that $\dens_G(h)>0$. Moreover, we deduce that there is an integer $h_0\mid n_0$ such that for every integer $t$ coprime to $n_0$ we have $\dens_G(th_0)>0$. We conclude by taking $t\equiv 1 \bmod n_0$ and $t\equiv ah_0^{-1}\bmod d$.
\end{proof}

The following result tells us in particular that for every prime number $\ell$ and for every $e\gg 0$ there is a positive density of primes $\mathfrak p$ of $K$ such that
$v_\ell(\ind_\p(G))=e$.

\begin{prop}
For every prime number $\ell$ there is some non-negative integer $e_\ell$ (and we can take $e_\ell=0$ for all but finitely many $\ell$) such that for every $e\geqslant e_\ell$ we have 
$$\dens_G(0,\ell^{e})>\dens_G(0,\ell^{e+1})\,.$$
Under GRH, for every $n>0$ and for every integer $z$, we have $\dens_G(z\ell^{e_\ell},\ell^n)>0$.
\end{prop}
\begin{proof}
By Chebotarev's density theorem 
the primes $\p$ of $K$ such that   $v_\ell(\ind_\p(G))=e$ have density $1/[K_{\ell^e,\ell ^e}:K]-1/[K_{\ell^{e+1},\ell ^{e+1}}:K]$, so the first assertion follows from the eventual maximal growth of the Kummer degrees, see \cite[Lem.\ 3.2]{PeruccaSgobba}.
By the first assertion (and by applying Theorem \ref{thm-positive} to the set $S$ of positive integers having $\ell$-adic valuation equal to $e$) for every $e\geq e_{\ell}$ there is some $b$ coprime to $\ell$ such that $\dens_G(b\ell^e)>0$. Then, for every prime $q$ coprime to $n_0$ we have $\dens_G(qb\ell^e)>0$, so we may conclude by selecting $q\equiv b^{-1} z\ell^{-v_\ell(z)}  \bmod \ell^n$, which is possible by Dirichlet's theorem on primes in arithmetic progressions.
\end{proof}

If $x,y$ are positive integers, then we use the notation $\gcd(x,y^\infty)$ to denote the positive integer obtained from $x$ by removing the prime factors that do not divide $y$.

\begin{thm}\label{thm1}
{\rm (under GRH)}.
We have $\dens_G(a,d)>0$ if and only if 
\[  \dens_G(a,\gcd(d,n_0^\infty))>0. \]
\end{thm}
\begin{proof}
Set $d_0=\gcd(d,n_0^\infty)$. The former inequality in the statement clearly implies the latter because the integers congruent to $a\bmod d$ are also congruent to $a \bmod d_0$. Now suppose that there is a positive density of primes $\p$ of $K$ such that $\ind_\p(G)\equiv a\bmod d_0$. From Theorem \ref{thm-positive} we deduce that there exists  $h\geq 1$ such that $h\equiv a\bmod d_0$ and $\dens_G(h)>0$. For 
every pair of positive integers $t,s$ coprime to $n_0$ such that $s\mid th$, we 
have $\dens_G(th/s)>0$.
If we choose $t\equiv s \bmod d_0$, then $th/s\equiv a\bmod d_0$. We claim that we may also choose $t,s$ so that $th/s\equiv a\bmod d/d_0$. Because of the Chinese remainder theorem it will be possible to simultaneously ensure that the two conditions hold, and hence $\dens_G(th/s)>0$ implies $\dens_G(a,d)>0$. To prove the claim, we first choose $t,s$ so that $\gcd(a,d/d_0)=\gcd(th/s,d/d_0)$, and then multiply $t$ by an integer invertible modulo $d/d_0$ to obtain the requested congruence.
\end{proof}

\begin{thm}
{\rm (under GRH)}.
The following conditions are equivalent:  
\begin{enumerate}
    \item the density $\dens_G(a,d)$ is positive; 
    \item there is an integer $A$ such that $A\equiv a\bmod d$ and $\dens_G(\gcd(A,n_0))$ is positive.
\end{enumerate}
\end{thm}

\begin{proof}
Write $D:=\lcm(d,n_0)$. By Theorem \ref{thm-positive} the density $\dens_G(a,d)$ is positive if and only if there is an integer $A\equiv a\bmod d$ for which $\dens_G(A,D)>0$. 
The latter holds if and only if there is an index $h$ such that $h\equiv A\bmod D$ and $\dens_G(h)>0$. Since $n_0\mid D$, we have $\gcd(h,n_0)=\gcd(A,n_0)$, and hence by Theorem \ref{thm-positive} we have that $\dens_G(h)$ is positive if and only if $\dens_G(\gcd(A,n_0))$ is positive. 
\end{proof}
Since the properties in (2) only depend on $A$ modulo $\lcm(d,n_0)$, we see
that it actually suffices to consider $A$ modulo  $\lcm(d,n_0)$.

\section{The Artin-type constants $B_{\chi}(r)$}
\label{EPconstants}
Let $r\ge 1$ be an integer. Recall the Euler product definition \eqref{Bchidef} of $B_{\chi}(r)$.
For $r=1$ this was introduced in \cite[Sec.\,6]{moreeFF} and denoted by 
$B_{\chi}$, along with a
variant $A_{\chi}$, where $p$ is restricted to those primes for which
$\chi(p)\ne 0$. We have
$$B_{\chi}(1)=A_{\chi}\prod_{p\mid d}\left(1-\frac{1}{p(p-1)}\right),$$
where $d$ is the modulus of the character. Note that $A_{\chi}=1$ in case $\chi$ is the 
principal character. 
\par If $\chi_0$ is the principal character, then $B_{\chi_0}(r)$ is a rational number. This 
leaves at most $\varphi(d')-1$ linearly independent Artin-type constants, with 
$d'=d/(a,d)$.
For example, in case $d'=3$ and $d'=4$ only one Artin-type constant is 
involved. They are real numbers. As an illustration we point out the result that the average density
of elements of multiplicative order $\pm 1\bmod 3$ equals
$\frac{5}{16}\pm \frac{3}{10}B_{\chi_3}(1)$, where $\chi_3$ 
is the non-principal character modulo $3$ and $B_{\chi_3}(1)=\frac{5}{6}A_{\chi_3}=
0.1449809353580\ldots$,
see \cite{moreeFF}. 
\par Approximating the numerical value of $B_{\chi}(r)$ by computing partial Euler products,
gives a quite poor accuracy. The following result allows us to do rather better and 
generalizes \cite[Thm.\,6]{moreeFF} to arbitrary $r$. 
It involves special values of Dirichlet L-series. Recall that for $\Re(s)>1$ and $\chi$ a 
Dirichlet character, we have
$$L(s,\chi)=\sum_{n=1}^{\infty}\frac{\chi(n)}{n^s}=\prod_p\Big(1-\frac{\chi(p)}{p^s}\Big)^{-1}.$$
\begin{thm}
\label{numeriek}
Let $p_1(=2),p_2,\ldots$ denote the sequence of consecutive primes and $\chi$ be any Dirichlet character. 
Put $$\Lambda_r=A_rL(r+1,\chi)L(r+2,\chi)L(r+3,\chi).$$
Then
$$B_{\chi}(r)=E_{r,n}\,\Lambda_r\,
\prod_{k=1}^n 
\left(1+\frac{\chi(p_k)}{ p_k(p_k^{r+1}-p_k^r-1)}\right)\Big(1-\frac{\chi(p_k)}{ p_k^{r+2}}\Big)
\Big(1-\frac{\chi(p_k)}{ p_k^{r+3}}\Big)$$
with
$$1-\frac{1}{ p_{n+1}^{r+2}}\le |E_{r,n}| \le 1+\frac{1}{ p_{n+1}^{r+2}},$$
provided that $r=1$ and $p_{n+1}\ge 5$, or $r=2$ and $p_{n+1}\ge 3$.
\end{thm}
\begin{proof} 
Recall the definition \eqref{constantAr} of $A_r$.
Noting that
$$(1-yt^{r+1})\left(\frac{1+\frac{(y-1)t^{r+1}}{ (1-yt^{r+1})(1-t)}}{ 1-\frac{t^{r+1}}{1-t}}\right)=
1+\frac{yt^{r+2}}{1-t-t^{r+1}},$$
we obtain 
\begin{equation}
    \label{Bproduct}
B_{\chi}(r)=A_r L(r+1,\chi)\prod_{k=1}^{\infty} 
\left(1+\frac{\chi(p_k)}{ p_k(p_k^{r+1}-p_k^r-1)}\right),
\end{equation}
on setting $y=\chi(p_k)$ and $t=\frac{1}{p_k}$.
We rewrite the infinite product  as
\[
L(r+2,\chi)L(r+3,\chi)\prod_{k=1}^{\infty} 
\left(1+\frac{\chi(p_k)}{ p_k(p_k^{r+1}-p_k^r-1)}\right)\Big(1-\frac{\chi(p_k)}{ p_k^{r+2}}\Big)
\Big(1-\frac{\chi(p_k)}{ p_k^{r+3}}\Big)
\]
in order to improve its convergence. Denoting the $k$-th term in the infinite product  
by $P_{r,k}$, we see that \eqref{Bproduct} holds with $E_{r,n}=\prod_{k\ge n+1}P_{r,k}$.
\par It remains to estimate the relative error $E_{r,n}$ (which
in general is a complex number). Multiplying out
$$(1-t-t^{r+1}+yt^{r+2})(1-yt^{r+2})(1-yt^{r+3})/(1-t-t^{r+1})$$
gives 
$$1+\frac{yt^{r+4}}{1-t-t^{r+1}}\Big(1+t^{r-1}+(1-y)t^r-yt^{r+2}-yt^{2r+2}+y^2t^{2r+3}\Big),$$
and leads to the estimate
$$  |P_{r,k}|\le 1+t^{r+3}G_r(t),$$
with 
$$G_r(t)=\frac{t(1+t^{r-1}+2t^r+t^{r+2}+t^{2r+2}+t^{2r+3})}{1-t-t^{r+1}}$$
and $t=\frac{1}{p_k}$. 
Note that $G_r(t)$ is increasing in $t$ and decreasing in $r$ in the region
$0<t<1$ and $r\ge 1$.
Thus $$|P_{r,k}|\le 1+p_k^{-r-3}G_r(p_k^{-1})\le 1+p_k^{-r-3}G_r(p_{n+1}^{-1})\quad \text{~for~every~}k\ge n+1.$$
As $t$ tends to zero, $G_r(t)$ tends to zero, and so we can choose $n$ so large that
$G_r(p_{n+1}^{-1})\le 1$.
Now $$|E_{r,n}|=\prod_{k\ge n+1}|P_{r,k}|<\prod_{p>p_n}\left(1+\frac{1}{ p^{r+3}}\right)<1+\sum_{m>p_n}\frac{1}{ m^{r+3}}.$$
Comparing the sum with an integral leads to the final estimate
$$|E_{r,n}|\le 1+\frac{1}{ p_{n+1}^{r+3}}+\int_{p_{n+1}}^{\infty}\frac{dz}{
z^{r+3}}\le 1+\frac{1}{ p_{n+1}^{r+2}},$$
where the sum is over the integers $m>p_n$.
Similarly, 
 $$|E_{r,n}|>\prod_{p>p_n}\left(1-\frac{1}{ p^{r+3}}\right)>1-\sum_{m>p_n}\frac{1}{ m^{r+3}}>1-\frac{1}{ p_{n+1}^{r+2}}.$$
Some calculus shows that $G_r(\frac{1}{p})\le 1$ if and only if $r=1$ and $p\ge 5$ or $r\ge 2$ and $p\ge 3$.
The proof is now completed on invoking the if-part of this statement.
\end{proof}
\begin{rem}
In the proof of \cite[Thm.\,6]{moreeFF} there are a few typos:\\
For ``$2+2t+t^3+t^5$'' read ``$2+2t+t^3+t^4+t^5$''.\\
For ``$t\ge 127$'' read ``$t\le 1/127$''.\\
For ``$p_{n+!}$'' read ``$p_{n+1}$''.
\end{rem}

 \section{Two examples}
 \label{sec:examples}
In this section we demonstrate our results by two relatively easy, but illustrative, examples for 
$K=\mathbb Q(\sqrt{5})$, $r=1$ and $d=5$. 
Some examples for the same $r$ and $d$ values, but with
$K=\mathbb Q$
are given in  Moree \cite[Table 2]{moreeII}.

\begin{table}[h!]
\centering\small
\begin{tabular}{c||ccccc}
\multirow{2}{*}{$G$}  & $\dens_G(0,5)$ &  $\dens_G(1,5)$ & $\dens_G(2,5)$ & $\dens_G(3,5)$ & $\dens_G(4,5)$ \\ [0.5ex]
  & $\frac{P_{0,5}(10^6)}{\pi_K(10^6)}$ & $\frac{P_{1,5}(10^6)}{\pi_K(10^6)}$ & $\frac{P_{2,5}(10^6)}{\pi_K(10^6)}$ & $\frac{P_{3,5}(10^6)}{\pi_K(10^6)}$ & $\frac{P_{4,5}(10^6)}{\pi_K(10^6)}$ \\ [1ex] 
\hline
\multirow{2}{*}{$\langle \frac{1+\sqrt{5}}2\rangle$}
 &  $0.100000$ & $0.418205$ & $0.296724$ & $0.0950872$ & $0.0899840$ \\
  & $0.100093$ & $0.419351$ & $0.296954$ & $0.0947177$ & $0.0888838$ \\[0.5ex] 
  \hline
  \multirow{2}{*}{$\langle  -\frac{5+\sqrt{5}}{2} \rangle$}
 & $0.100000$ & $0.451872$ & $0.266393$ & $0.0995570$ & $0.0821785 $    \\ 
  & $0.099787$ & $0.450979$ & $0.267518$ & $0.0996599$ & $0.0820564$ \\[0.5ex] 
\end{tabular}
\caption{Examples of densities $\dens_G(a,5)$ with $K=\Q(\sqrt{5})$}
\label{table1}
\end{table}

In Table \ref{table1} we denote by $P_{a,d}(x)$ the number of primes $\p$ of $K$ of norm 
up to $x$ such that $\ind_\p(G)\equiv a\bmod d$, and by $\pi_K(x)$ the number of primes $\p$ of $K$ with norm up to $x$.
The top row gives the theoretical density, the second row an experimental approximation (both with rounding
of the final decimal).

\par We will now treat these
two examples without using the machinery of Section \ref{sec:APdensity} (however, 
with complicated enough examples this becomes unavoidable). Our approach requires some further notation.
Given a divisor $\delta$ of an integer $d_1$, we put
   $$\rho_{\delta,d_1}(a,d):=\sum_{t\equiv a\bmod d}\ \sum_{\substack{v\geq1\\ (v,d_1)=\delta}}\frac{\mu(v)}{vt\varphi(vt)}.$$
\subsection{First example}
\begin{prop}
\label{prop:first}
Set $K=\Q(\sqrt{5})$ and $G=\langle \frac{1+\sqrt{5}}2\rangle$.
We have $$\dens_G(0,5)=\frac{1}{10}=2\rho(0,5)$$
and, for $1\le a\le 4$, assuming GRH,
$$\dens_G(a,5)=\frac{18}{19}\rho(a,5)=\frac{9}{38}\Big(\frac{19}{20}+\overline{\psi(a)}B_{\psi}(1)+\psi(a)B_{\psi^3}(1)
+\psi^2(a)B_{\psi^2}(1)\Big).$$
\end{prop}
\begin{proof} The first claim follows from \eqref{simpledegree} and Proposition \ref{null}. 
Next assume that $1\le a\le 4$. 
We have $\rho(a,5)=\rho_{1,5}(a,5)+\rho_{5,5}(a,5)$.
If $5\nmid t$, then
$$\sum_{5\mid v}\frac{\mu(v)}{vt\varphi(vt)}=\sum_{w}\frac{\mu(5w)}{5wt\varphi(5wt)}=
-\frac{1}{20}\sum_{5\nmid w}\frac{\mu(w)}{wt\varphi(wt)}.$$
We conclude that $\rho_{5,5}(a,5)=-\frac{1}{20}\rho_{1,5}(a,5)$. 
It thus follows that $\rho_{1,5}(a,5)=\frac{20}{19}\rho(a,5)$ and $\rho_{5,5}(a,5)=-\frac{1}{19}\rho(a,5)$.
Since the degree $[K_{n,n}:K]$ equals $\varphi(n)n$ if $5\nmid n$ and $\frac{1}{2}n\varphi(n)$ otherwise, we 
infer that
$\dens_G(a,5)=\rho_{1,5}(a,5)+2\rho_{5,5}(a,5)=\frac{18}{19}\rho(a,5)$. The proof is completed on invoking 
    \eqref{rhoB} and noting that $\psi^2(a)$ is real and $\overline{\psi^3(a)}=\psi(a)$.
\end{proof}
\par Approximations to $B_{\chi}(1)$ can be 
found in Table \ref{table2},
where $\psi$ denotes the character modulo $5$ determined uniquely by
$\psi(2)=i$.

\begin{table}[h!]
\centering\small
\begin{tabular}{c||c}
$\chi$  & $B_{\chi}(1)$ \\ [0.5ex] 
\hline
$\psi$ & $0.34645514515465\ldots+i\cdot 0.21283903970350\ldots$    \\
$\psi^2$ & $0.12284254160167\ldots$    \\
$\psi^3$ & $0.34645514515465\ldots-i\cdot 0.21283903970350\ldots$    \\
$\psi^4$ & $0.95$    \\ [0.5ex] 
\end{tabular}
\caption{The constants $B_{\chi}(1)$ for $d=5$}
\label{table2}
\end{table}

The character group has $\psi,\psi^2,\psi^3$ and $\psi^4$ as 
elements, with $\psi^4$ being the principal character. 
The table
is taken from \cite[Table 3]{moreeFF}, where 
for $d\le 12$ further approximations can be found. It was kindly
verified by Alessandro Languasco using Theorem \ref{numeriek} with $n=10^6$.

\subsection{Second example}
\begin{prop}
Set $K=\Q(\sqrt{5})$ and $G=\langle  -\frac{5+\sqrt{5}}{2} \rangle$.
Let $1\le a\le 4$.
One of $a$ and $a+5$ is even. Denoting this number by $a_1$, assuming GRH, we have
$$\dens_G(a,5)=\frac{20}{19}\rho(a,5)-\frac{4}{19}\rho(a_1,10).$$
Furthermore, $\dens_G(0,5)=\frac{1}{10}$.
\end{prop}
\begin{proof}
Using \eqref{simpledegree} we see that $\dens_G(0,5)=\frac{1}{10}$. We will determine $\dens_G(a,10)$ in 
case $5\nmid a$. The result then follows on adding $\dens_G(a,10)$ and $\dens_G(a+5,10)$.
\par Since $\Q\Big(\sqrt{-\frac{5+\sqrt{5}}{2}}\Big)=\Q(\zeta_5)$,
the degree $[K_{n,n}:K]$ equals $\varphi(n)n$ if $5\nmid n$, it equals $\frac{1}{2}n\varphi(n)$ if $(n,10)=5$, and it equals $\frac{1}{4}n\varphi(n)$ if $10\mid n$. 
These degree considerations lead to 
$$
\dens_G(a,10)=
\begin{cases}
\rho_{1,5}(a, 10)+4\rho_{5,5}(a, 10) & \text{ if } 2\mid a;\\
\rho_{1,5}(a, 10)+2\rho_{5,10}(a,10)+4\rho_{10,10}(a, 10) & \text{ if } 2\nmid a.
\end{cases}
$$
\par Reasoning as in the proof of 
Proposition \ref{prop:first} we deduce that $\rho_{5,5}(a,10)=-\frac{1}{20}\rho_{1,5}(a,10)$ and 
$\rho_{1,5}(a,10)=\frac{20}{19}\rho(a,10)$. It follows that
$\dens_G(a,10)=\frac{4}{5}\rho_{1,5}(a,10)=\frac{16}{19}\rho(a,10)$ in case $a$ is even.
\par If $a$ is odd, then so are the 
integers $t\equiv a\bmod 10$ and so $\rho_{10,10}(a,10)=-\frac{1}{2}\rho_{5,10}(a,10)$,
leading to $\dens_G(a,10)=\rho_{1,5}(a,10)$.
Reasoning as in the proof of 
Proposition \ref{prop:first}, we then deduce that $\dens_G(a,10)=\frac{20}{19}\rho(a,10)$.
\end{proof}
For reasons of space we refrain here from explicitly writing out $\dens_G(a,5)$ as a linear sum
in the $B_{\chi}$'s, but we will indicate
how this is done. For $\rho(a,5)$ we use \eqref{rhoB}. For $a$ with 
$5\nmid a$ we have by \cite[Prop.~6]{moreeFF} with $w=5$ and $\delta=2$,
$$\rho(2a,10)=\frac{3}{8}\sum_{\chi \bmod 5}\overline{\chi(a)}\frac{B_{\chi}(1)}{2+\chi(2)}\,.$$

\section*{Acknowledgments}
This project was started when the first author gave a talk in the Luxembourg Number Theory Day 2019.
He thanks the other authors for the invitation and for several subsequent invitations. The second and third author thank the Max Planck Institut f\"ur Mathematik and the first author for organizing a short visit in October 2022.
Thanks are also due to Alessandro Languasco for verifying Table \ref{table2} using Theorem \ref{numeriek}. 
We thank Valentin Blomer for pointing out reference \cite{ambrose}.

\end{document}